\documentclass[12pt,final]{article}
\usepackage{amsmath,amsthm,amsfonts,amssymb,graphicx,hyperref,enumerate,psfrag}
\usepackage{fullpage}

\usepackage[utf8]{inputenc} 

\newtheorem{lemma}{Lemma}
\newtheorem{obs}{Observation}

\newtheorem{proposition}{Proposition}
\newtheorem{theorem}{Theorem}
\newtheorem{conjecture}[lemma]{Conjecture}

\newtheorem{assumption}{Assumption}

\newcommand{\EE}{{\mathbb{E}}}

\newcommand{\PP}{\mathbb{P}}
\newcommand{\Z}{\mathbb {Z}}

\newcommand{\cH}{\mathcal {H}}

\newcommand{\cE}{\mathcal {E}}

\newcommand{\ee}{ \delta}

\newcommand{\pp}{{{p}}}

\newcommand{\ent}{{\mathrm{Ent}}}

\title{Entropy dissipation estimates for inhomogeneous zero-range processes}
\author{Jonathan Hermon, Justin Salez}

\begin{document}
\maketitle

\begin{abstract}
Introduced by Lu \& Yau (CMP, 1993), the martingale decomposition method is a powerful recursive strategy that has produced sharp log-Sobolev inequalities for homogeneous particle systems. However, the intractability of certain covariance terms has so far precluded  applications to heterogeneous models. Here we demonstrate that the existence of an appropriate coupling can be exploited to bypass this limitation effortlessly. Our main result is a dimension-free modified log-Sobolev inequality  for zero-range processes on the complete graph, under the only requirement that all rate increments  lie in a compact subset of $(0,\infty)$. This settles an open problem raised by Caputo \& Posta (PTRF, 2007) and reiterated by  Caputo, Dai Pra \& Posta (AIHP, 2009). We believe that our  approach is simple enough to be applicable to many systems.
\end{abstract}
\tableofcontents

\newpage
\section{Introduction}
\subsection{Entropy dissipation estimates}

 Consider a reversible Markov generator $Q$ with respect to some probability distribution $\pi$ on a finite state space  $\Omega$. In other words, $Q$ is a $\Omega\times\Omega$ matrix with non-negative off-diagonal entries, with each row summing up to $0$, and satisfying the \emph{local balance} equations
\begin{eqnarray}
\label{reversible}
\pi(x)Q(x,y) & = & \pi(y)Q(y,x),
\end{eqnarray}
for all $x,y\in \Omega$. When $Q$ is irreducible,  the Markov semi-group $P_t=e^{tQ}$ generated by $Q$ \emph{mixes}:  for any observable $f\colon\Omega\to(0,\infty)$, we have the pointwise convergence
\begin{eqnarray}
\label{mixing}
P_t f & \xrightarrow[t\to\infty]{} &  \EE[f(X)],
 \end{eqnarray} 
where throughout the paper, $X$ denotes a $\pi-$distributed random variable. A natural way to quantify this convergence consists in measuring the rate at which the \emph{entropy} 
\begin{eqnarray}
\ent (f) & := & \EE \left[f(X)\log f(X)\right]-\EE [f(X)]\log \EE [f(X)],
\end{eqnarray}
decays along the semi-group. Specifically, one looks for a constant $\alpha>0$, as large as possible, such that for all observables $f\colon \Omega\to(0,\infty)$ and all times $t\ge 0$,
\begin{eqnarray}
\label{entropy}
\ent \left(P_t f\right) & \le & e^{-\alpha t}\,\ent (f).
\end{eqnarray}
The optimal value of  $\alpha$ is called the \emph{entropy dissipation constant} and will be denoted by $\alpha(Q)$. Writing $\cE (f,g) := -\EE \left[f(X)(Q g)(X)\right]$ for the underlying Dirichlet form, we compute
\begin{eqnarray}
\frac{d}{dt}\,\ent \left(P_t f\right)  & = & -\cE \left(P_tf,\log P_t f\right).
\end{eqnarray}
 Thus, $\alpha(Q)$ is more effectively characterized as the largest constant $\alpha>0$ such that the following \emph{modified log-Sobolev inequality} (\textsc{MLSI}) holds: for all observables $f\colon \Omega\to(0,\infty)$,
\begin{eqnarray}
\label{MLSI}
\cE  \left(f,\log f\right) &  \ge & \alpha\,\ent \left(f\right).
\end{eqnarray}
 We refer to the tutorial paper \cite{MR2283379} or the textbook \cite{MR2341319} for more details on this fundamental functional inequality and its relation to hypercontractivity, concentration and mixing times. For intrinsic reasons exposed in \cite{MR3843561}, establishing sharp \textsc{MLSI}'s for Markov chains on finite spaces remains a notoriously challenging task. A natural and important context where such entropic estimates have received a particular attention is that of interacting particle systems and, in particular, \emph{zero-range} processes.

\subsection{Zero-range dynamics}

Introduced by Spitzer \cite{Spitzer}, the \emph{zero-range process} (\textsc{ZRP}) is a  generic  conservative particle system in which individual jumps occur at a rate which only depends on the source, the destination, and the number of particles present at the source. We shall here focus on the mean-field version of the model, which is parameterized by the following ingredients:
\begin{itemize}
\item two integers $n,m\ge 1$ representing the numbers of sites and  particles, respectively; 
\item a  probability vector $\pp=(\pp_1,\ldots,\pp_n)$ specifying the distribution of the jumps;
\item a function $r_i\colon\{1,2,\ldots\}\to(0,\infty)$ encoding the kinetics at site $i\in [n]:=\{1,\ldots,n\}$.
\end{itemize}
The \textsc{ZRP} with these parameters is a continuous-time Markov chain on the state space 
\begin{eqnarray}
\Omega & := & \left\{x\in\Z_+^n\colon \sum_{i=1}^n x_i=m\right\},
\end{eqnarray}
where $x_i$ represents the number of particles at site $i$. The action of the generator is given by 
\begin{eqnarray}
\label{def:markov}
(Q f)(x) & := & \sum_{i=1}^n\sum_{j=1}^n r_i(x_i)\pp_j\left(f(x+\ee_j-\ee_i)-f(x)\right),
\end{eqnarray}
where $(\ee_1,\ldots,\ee_n)$ denotes the canonical $n-$dimensional basis, and with the convention that $r_i(0)=0$ for all $i\in [n]$ (no jump from empty sites). This generator  is easily seen to be reversible with respect to the  explicit distribution
\begin{eqnarray}
\label{statio}
\pi(x) & \propto & \prod_{i=1}^n\frac{ \pp_i^{x_i}}{r_i(1)\cdots r_i(x_i)},
\end{eqnarray}
with the normalization being chosen so that $\pi$ is a probability distribution on $\Omega$. A somewhat degenerate but instructive situation is obtained with the linear choice
\begin{eqnarray}
\label{linear}
 r_i(\ell) & = & c\times\ell,
 \end{eqnarray} for all $i\in[n]$ and $\ell\ge 0$, where $c>0$ is an arbitrary constant. In this case, the model trivializes in the sense that the $m$ particles perform independent random walks, each jumping at rate $c$ according to $\pp$. Thanks to  the tensorization property of variance and entropy, the Poincaré and modified log-Sobolev constants of the whole system then reduce to those of a single particle, which are easily seen to be bounded away from $0$ independently of $\pp$. In light of this, it is reasonable to expect that  dimension-free functional inequalities will persist in the perturbative regime where the rate functions are \emph{nearly linear}.  Turning this intuition into rigorous estimates has been and continues to be a subject of active research, see e.g., \cite{MR1415232,MR1681098,MR2073330,MR2200172,MorrisZRP,HS} on the Poincaré side, and \cite{MR2184099,Cap,Graham,MR2548501,MR3513606} on the log-Sobolev side.

\subsection{Result and related works}
Perhaps the most natural way  to formalize the idea that the rate functions should be   \emph{nearly linear} consists in requiring that their increments all lie in a fixed  compact subset of $(0,\infty)$.
\begin{assumption}\label{assume:rates}
There are constants $\Delta,\delta>0$ such that for every $i\in [n]$ and every $\ell\in\Z_+$,
\begin{eqnarray}
\delta \  \le & r_i(\ell+1)-r_i(\ell) & \le  \ \Delta.
\end{eqnarray}
\end{assumption}
In a remarkable work \cite{MR2548501}, Caputo, Dai Pra \& Posta developed a novel method, based on the so-called Bochner-Bakry-Emery approach introduced in \cite{MR889476}, to establish a dimension-free $\textsc{MLSI}$ in the perturbative regime where the ratio $\Delta/\delta$  is sufficiently small.
\begin{theorem}[Caputo, Dai Pra \& Posta, \cite{MR2548501}]\label{th:CPP}Suppose that Assumption \ref{assume:rates} holds with
\begin{eqnarray}
\label{restrict}
\frac{\Delta}{\delta} & < & 2,
\end{eqnarray}
and take $\pp=(\frac{1}{n},\ldots,\frac{1}{n})$. Then, for any number $m$ of particles,  the \textsc{ZRP}  satisfies 
\begin{eqnarray}
\label{CPP}
\alpha(Q) & \ge &   2\delta-\Delta.
\end{eqnarray}
\end{theorem}
In the so-called \emph{homogeneous setting} where the rate function $r_i$ is not allowed to depend on $i$, Assumption \ref{assume:rates} can be relaxed at the price of considerable efforts, see \cite{Cap}. However, the martingale approach used therein does not extend to inhomogeneous models in a natural way and, to the best of our knowledge, Theorem \ref{th:CPP} constitutes the only available criterium  for inhomogeneous rates.
Although the bound (\ref{CPP}) trivializes as $\Delta$ approaches the threshold $2\delta$, the authors predicted the persistence of a dimension-free $\textsc{MLSI}$ beyond the perturbative regime (\ref{restrict}).  Specifically, they formulated the following conjecture,  reiterated in \cite{Cap}. 
\begin{conjecture}[Caputo, Dai Pra \& Posta, \cite{MR2548501}]\label{conj}
Under Assumption \ref{assume:rates}, there is a dimension-free constant $c(\delta,\Delta)>0$ such that for $\pp=(\frac{1}{n},\ldots,\frac{1}{n})$ and any number $m$ of particles, 
\begin{eqnarray}
\alpha(Q) & \ge &   c(\delta,\Delta).
\end{eqnarray}
\end{conjecture}
Proving this requires new ideas, since it was  noted in \cite{MR2548501} that the  convexity of $t\mapsto \ent(P_tf)$ fails as $\frac{\Delta}{\delta}$ gets large, making the Bochner-Bakry-Emery approach unapplicable. In the present paper, we establish the following strengthening of Conjecture \ref{conj}. 
\begin{theorem}[Dimension-free \textsc{MLSI}]
\label{th:ZRP}
Under the sole Assumption \ref{assume:rates}, and for any choice of the parameters $\pp$ and $m$,  the modified log-Sobolev constant of the \textsc{ZRP}  satisfies 
\begin{eqnarray}
\alpha(Q) & \ge &   \frac{\delta^2}{2\Delta}.
\end{eqnarray}
\end{theorem}
This estimate is sharp up to a factor $2$, as can already been seen in the  linear case (\ref{linear}) (see, \cite[Example 3.10]{MR2283379}). More importantly, our main contribution lies in the simplicity  of the method used to prove Theorem \ref{th:ZRP}. Our starting point is the following elementary observation, which is a straightforward consequence of the  product form (\ref{statio}). 
\begin{obs}[Recursive structure]\label{obs:rec} If $X=(X_1,\ldots,X_n)$ has law $\pi$, then for any site $i\in[n]$, the conditional law of $(X_1,\ldots,X_{i-1},X_{i+1},\ldots,X_n)$ given $X_i$ coincides with the stationary law of a new \textsc{ZRP} with  $n-1$ sites, $m-X_i$ particles, rates $(r_1,\ldots,r_{i-1},r_{i+1},\ldots,r_n)$ and probability vector 
$
\left(\frac{\pp_1}{1-\pp_i},\ldots,\frac{\pp_{i-1}}{1-\pp_i},\frac{\pp_{i+1}}{1-\pp_i},\ldots,\frac{\pp_n}{1-\pp_i}\right).
$
\end{obs}
Since the new \textsc{ZRP}  inherits Assumption  \ref{assume:rates}  from the original one, a natural approach to Theorem \ref{th:ZRP} consists in proceeding by induction over the dimension $n$. This is in fact a classical  strategy for establishing  functional inequalities, known as the \emph{martingale decomposition method}. Introduced  by Lu \& Yau \cite{MR1233852} in the context of Kawasaki and Glauber dynamics, it has been successfully applied to various  interacting particle systems \cite{MR1483598,MR1675008,MR2023890,Cap}, as well as other Markov chains enjoying an appropriate recursive structure \cite{MR1944012,1181997,MR2099650,2019arXiv190202775H}. To convert the above observation into an effective functional inequality however, one needs to estimate the Dirichlet form of the $n-$dimensional system in terms of that of the $(n-1)-$dimensional system. Because of the interaction between particles, this decomposition inevitably produces certain \emph{cross terms}, whose intractibility has so far precluded  applications to heterogeneous models. Our main contribution consists in showing that the existence of an appropriate coupling -- see Proposition \ref{pr:coupling} below -- allows one to bypass this limitation effortlessly. We firmly believe that our idea is simple enough to be applicable to many other settings.

\section{Proofs}

From now onwards, we suppose that Assumption \ref{assume:rates} is satisfied.

\subsection{Monotone coupling}
Using the fact that each rate function $r_i\colon\{1,2,\ldots\}\to(0,\infty)$ is non-decreasing, we may  {enrich} our \textsc{ZRP}  by adding a \emph{tagged particle} on top of it, whose position $(J(t)\colon t\ge 0)$ evolves as follows: conditionally on the background process $(X(t)\colon t\ge 0)$ being currently in some state $x\in\Omega$, we let the tagged particle jump across $[n]$ according to the Markov generator 
\begin{eqnarray}
\label{tagged}
L_x(i,j) & := & \left(r_{i}(x_i+1)-r_i(x_i)\right)\left(\pp_j-\bf 1_{(i=j)}\right).
\end{eqnarray}
In more formal terms, we consider a  Markov process $\left((X(t),J(t))\colon t\ge 0\right)$ taking values in the product space $\widehat{\Omega}:=\Omega\times [n]$, and evolving   according to the generator 
\begin{eqnarray}
\widehat{Q}\left((x,i),(y,j)\right) & := & Q(x,y){\bf 1}_{(i=j)} + L_x(i,j){\bf 1}_{(x=y)}.
\end{eqnarray}
We refer to this process as the \emph{tagged} \textsc{ZRP}. 
An elementary but crucial observation is that 
\begin{enumerate}[(i)]
\item the first-coordinate $(X(t)\colon t\ge 0)$ is  a \textsc{ZRP} with parameters $(n,m,\pp, r)$;
\item the aggregated process $\left(X(t)+\ee_{J(t)}\colon t\ge 0\right)$ is   a \textsc{ZRP} with parameters $(n,m+1,\pp,r)$. 
\end{enumerate}
The existence of such a  monotone coupling  between zero-range processes with different numbers of particles  is of course well known, and has been extensively used in the past. However, its consequences on the martingale approach do not seem to have been explored, and this is where our conceptual contribution lies. Specifically, our interest will here reside in the invariant law $\widehat{\pi}$  of $\widehat{Q}$, which is uniquely determined by the \emph{global balance} equations
\begin{eqnarray}
\label{balance}
\sum_{x\in\Omega}\widehat{\pi}(x,j)Q(x,y)+\sum_{i=1}^n\widehat{\pi}(y,i)L_y(i,j) & = & 0,
\end{eqnarray}
for all $(y,j)\in\widehat{\Omega}$. Although we do not have any explicit expression for $\widehat{\pi}$, we note that its first marginal has to be invariant under $Q$ by construction, and is therefore simply the law $\pi$ defined at (\ref{statio}). We also note that, in the linear case (\ref{linear}),  the generator $L_x$ of the tagged particle becomes independent of the background state $x\in\Omega$, resulting in the product form
\begin{eqnarray}
\widehat{\pi}(x,j) & = & \pi(x)\pp_j.
\end{eqnarray}
The next lemma states that under Assumption  \ref{assume:rates}, $\widehat{\pi}$ is  not far from this product measure. 
\begin{lemma}[Product-measure approximation]\label{lm:product}The invariant law  of the tagged $\textsc{ZRP}$   satisfies 
\begin{eqnarray}
\label{coupling:claim}
\max_{(x,j)\in\widehat{\Omega}}\left\{\frac{\widehat{\pi}(x,j)}{\pi(x)\pp_j}\right\} & \le & \frac{\Delta}{\delta}.
\end{eqnarray}

\end{lemma}
\begin{proof}Fix $j\in [n]$ and choose $y\in\Omega$ such that
\begin{eqnarray}
\label{maximal}
\frac{\widehat{\pi}(y,j)}{\pi(y)} & = & \max_{x\in\Omega}\left\{\frac{\widehat{\pi}(x,j)}{\pi(x)} \right\}.
\end{eqnarray}
At the point $(y,j)$, the first sum in the global balance equation $(\ref{balance})$ is non-positive. Indeed,
\begin{eqnarray}
\sum_{x\in\Omega}\widehat{\pi}(x,j)Q(x,y) & \le & \frac{\widehat{\pi}(y,j)}{\pi(y)}\sum_{x\in \Omega}\pi(x)Q(x,y) \ = \ 0,
\end{eqnarray}
where the equality is simply the balance equation $\pi Q=0$ at state $y$. Consequently, the second sum in (\ref{balance}) must be non-negative. In other words, 
\begin{eqnarray*}
0 & \le & \sum_{i=1}^n\widehat{\pi}(y,i)L_y(i,j) \\
& \le & \sum_{i=1}^n\widehat{\pi}(y,i)\left(\Delta\pp_j -\delta {\bf 1}_{(i=j)}\right) \ = \ \Delta\pp_j \pi(y)-\delta\widehat{\pi}(y,j),
\end{eqnarray*}
where we have used   Assumption (\ref{assume:rates}), and then the  fact that the first marginal of $\widehat{\pi}$ is $\pi$.  \end{proof}
These considerations lead us to the following result, which will constitute our main tool.
\begin{proposition}[Main coupling]\label{pr:coupling}Given a site $i\in[n]$, we may jointly construct a $\pi-$distributed random variable $X$ and a $[n]\setminus\{i\}-$valued random variable $J$ in such a way that 
\begin{eqnarray}
\label{representation}
\mathrm{Law}\left(X-\delta_i+\delta_J|X_i=\ell\right) & = & \mathrm{Law}\left(X|X_i=\ell-1\right),
\end{eqnarray}
for each level $\ell\in\{1,\ldots,m\}$ and moreover, for every $(x,j)\in \Omega\times ([n]\setminus\{i\})$,
\begin{eqnarray}
\label{product}
\PP\left(J=j|X=x\right) & \le &  \frac{\Delta\pp_j}{\delta(1-\pp_i)}.
\end{eqnarray} 
\end{proposition}

\begin{proof}We simply let $X_i$ be distributed according to the $i-$th marginal of $\pi$ and, given that $X_i=\ell$, we let $\left((X_1,\ldots,X_{i-1},X_{i+1},\ldots,X_n),J\right)$ be distributed according to the stationary distribution of the $(n-1)-$dimensional tagged $\textsc{ZRP}$ with   $m-\ell$ particles, rates $(r_1,\ldots,r_{i-1},r_{i+1},\ldots,r_n)$ and probability vector 
$
\left(\frac{\pp_1}{1-\pp_i},\ldots,\frac{\pp_{i-1}}{1-\pp_i},\frac{\pp_{i+1}}{1-\pp_i},\ldots,\frac{\pp_n}{1-\pp_i}\right).
$
Observation \ref{obs:rec} ensures that the random variable $X:=(X_1,\ldots,X_n)$ is  distributed according to $\pi$.  Moreover, property $(\ref{product})$ is guaranteed by Lemma \ref{lm:product}. Finally,  the identity (\ref{representation}) is clear since under both distributions, the $i-$th coordinate is almost-surely equal to $\ell-1$ while the joint law of the remaining coordinates is the same, thanks to observation (ii) above. 
\end{proof}

\subsection{Single-site estimate}

In this section, we show how the above coupling  implies -- without effort -- the crucial local  \textsc{MLSI} needed for our inductive proof of Theorem \ref{th:ZRP}. We start by introducing some notation. Let us decompose the Dirichlet form of the process as 
\begin{eqnarray}
\label{dec:E}
\cE(f,g) & = & \frac{1}{2}\sum_{i=1}^n\sum_{j=1}^n\cE_{ij}(f,g),
\end{eqnarray}
where $\cE_{ij}(f,g)$ captures the contribution from all jumps with source $i$ and destination $j$, i.e.
\begin{eqnarray}
\cE_{ij}(f,g) & := & \pp_j\EE\left[r_i(X_i)\left(f(X-\delta_i+\delta_j)-f(X)\right)\left(g(X-\delta_i+\delta_j)-g(X)\right)\right].
\end{eqnarray}
Given an observable $f\colon \Omega\to (0,\infty)$ and a site $i\in[n]$, we define $f_i\colon\{0,\ldots,m\}\to(0,\infty)$ by
\begin{eqnarray}
f_i(\ell) & := & \EE\left[f(X)|X_i=\ell\right].
\end{eqnarray}
With this notation in hands, our aim is to establish the following local estimate, which relates the Dirichlet contribution from a single site $i$ to the conditional entropy   given $X_i$.  
\begin{proposition}[Local \textsc{MLSI}]\label{pr:main}For any observable $f\colon\Omega\to(0,\infty)$ and any site $i\in[n]$,  
\begin{eqnarray*}
 \EE\left[f_i(X_i)\log f_i(X_i)\right]-\EE[f(X)]\log\EE[f(X)]   & \le & \frac{\Delta}{\delta^2}\sum_{j\in[n]\setminus\{i\}}\cE_{ij}(f,\log f).
\end{eqnarray*}
\end{proposition}
\begin{proof}Consider the bivariate function $\cH\colon(0,\infty)^2\to[0,\infty)$ defined by
\begin{eqnarray}
\label{def:H}
\cH(u,v) & := &  
(u-v)\left(\log u-\log v\right).
\end{eqnarray}
This function  is convex, because its hessian matrix
\begin{eqnarray}
\left[
\begin{array}{cc}
\partial_{uu}\cH & \partial_{uv}\cH  \\
\partial_{vu}\cH  & \partial_{vv}\cH
\end{array}
\right]
& = & \frac{u+v}{(uv)^2}\left[
\begin{array}{cc}
v^2 & -uv  \\
-uv & u^2
\end{array}
\right]
\end{eqnarray}
is positive semi-definite at each point $(u,v)\in(0,\infty)^2$.
Now, fix a site $i\in[n]$, and let $(X,J)$ be as in Proposition \ref{pr:coupling}. On the event $\{X_i\ge 1\}$, the definition of $f_i$ and Property (\ref{representation}) imply
\begin{eqnarray}
f_i(X_i) & = & \EE[f(X)|X_i],\\
f_i(X_i-1) & = & \EE[f(X+\delta_J-\delta_i)|X_i].
\end{eqnarray}
Thus, the pair $\left(f_i(X_i),f_i(X_i-1)\right)$ is the conditional expectation of $\left(f(X),f(X-\delta_i+\delta_J)\right)$ given $X_i$. By the conditional Jensen inequality, we deduce that on the event $\{X_i\ge 1\}$,
\begin{eqnarray}
\cH\left(f_i(X_i),f_i(X_i-1)\right)  & \le &  \EE\left[\cH\left(f(X),f(X-\delta_i+\delta_J)\right)|X_i \right].
\end{eqnarray}
Multiplying by $(1-\pp_i)r_i(X_i)$ and taking expectations, we obtain
\begin{eqnarray*}
(1-\pp_i)\EE\left[r_i(X_i)\cH\left(f_i(X_i),f_i(X_i-1)\right)\right] & \le & (1-\pp_i)\EE\left[r_i(X_i)\cH\left(f(X),f(X-\delta_i+\delta_J)\right)\right]\\
& \le & \frac{\Delta}{\delta}\sum_{j\in[n]\setminus\{i\}}\cE_{ij}(f,\log f),
\end{eqnarray*}
where the second line uses Property (\ref{product}). Thus, our task boils down to proving
\begin{eqnarray*}
\EE\left[f_i(X_i)\log f_i(X_i)\right]-\EE[f(X)]\log\EE[f(X)] & \le & \frac{
(1-\pp_i)}{\delta}\EE\left[r_i(X_i)\cH\left(f_i(X_i),f_i(X_i-1)\right)\right].
\end{eqnarray*}
Since $\EE[f_i(X_i)]=\EE[f(X)]$, the left-hand side is the entropy of $f_i$ with respect to the law of $X_i$. On the other hand, the right-hand side is exactly $\frac{1}{\delta}\cE_K(f_i,\log f_i)$, where $\cE_K$ denotes the Dirichlet form of the birth-and-death generator $K$ on $\{0,\ldots,m\}$ defined by
\begin{eqnarray}
\forall \ell\in\{1,\ldots,m\},\qquad K(\ell,\ell-1) & := & (1-\pp_i)r_i(\ell)\\
\forall \ell\in\{0,\ldots,m-1\},\qquad K(\ell,\ell+1) & := & \pp_i\sum_{j\in[n]\setminus \{i\}}\EE\left[\left.r_j(X_j)\right|X_i=\ell\right].
\end{eqnarray}
Note that these rates are reversible with respect to  the law of $X_i$, because they mimic the zero-range dynamics at site $i$. In terms of the generator $K$, the above claim reads 
\begin{eqnarray}
\label{BDC}
\alpha(K) & \ge & \delta.
\end{eqnarray} 
It remains to establish this one-dimensional \textsc{MLSI}. By Assumption  \ref{assume:rates}, the death rates satisfy
\begin{eqnarray*}
K(\ell+1,\ell)-K(\ell,\ell-1) & = &  (1-\pp_i)\left(r_i(\ell+1)-r_i(\ell)\right) \ \ge \ (1-p_i)\delta.
\end{eqnarray*}
Regarding the birth rates, we may invoke Property (\ref{representation}) to write
\begin{eqnarray*}
K(\ell-1,\ell)-K(\ell,\ell+1) & = & \pp_i\sum_{j\in[n]\setminus \{i\}}\EE\left[\left.r_j(X_j+{\bf 1}_{(J=j)})-r_j(X_j)\right|X_i=\ell\right] \ \ge \ 
\pp_i\delta.
\end{eqnarray*}
For birth-and-death chains,  these uniform bounds on the rate increments are known to  imply the \textsc{MLSI} (\ref{BDC}), see e.g., \cite[Theorem 3.1]{MR2548501} or \cite[Lemma 2.2]{Cap}. 
\end{proof}

\subsection{The induction argument}
We are finally in position to prove Theorem \ref{th:ZRP} by induction over $n$. The claim is trivial for $n=1$. We now assume that $n\ge 2$ and that the claim is already 
 proven for $(n-1)-$dimensional systems. Fix an  observable $f\colon\Omega\to(0,\infty)$ and a site $i\in [n]$, and consider the $(n-1)-$dimensional $\textsc{ZRP}$ obtained by conditioning on $X_i$, as in Observation \ref{obs:rec}. Viewing $f$ as a function of the $n-1$ remaining variables, the induction hypothesis ensures that
\begin{eqnarray}
\ent(f|X_i) & \le & \frac{2\Delta}{\delta^2}\sum_{j,k\in [n]\setminus\{i\}}\frac{\pp_j}{1-\pp_i}\EE\left[\left.r_k(X_k)\cH\left(f(X),f(X+\ee_j-\ee_k)\right)\right|X_i \right],
\end{eqnarray} 
where the entropy  on the left-hand side is computed conditionally on $X_i$, i.e.
\begin{eqnarray}
\ent(f|X_i) & := & \EE\left[\left.f(X)\log f(X)\right|X_i \right]-f_i(X_i)\log f_i(X_i).
\end{eqnarray}
Taking expectations, we arrive at
\begin{eqnarray}
\EE[f(X)\log f(X)] & \le & \EE[f_i(X_i)\log f_i(X_i)]+\frac{\Delta}{\delta^2(1-\pp_i)}\sum_{j,k\in [n]\setminus\{i\}}\cE_{jk}(f,\log f).
\end{eqnarray}
We now multiply by $1-\pp_i$ and use Proposition \ref{pr:main} to obtain
\begin{eqnarray}
(1-p_i)\ent(f) & \le & \frac{\Delta}{\delta^2}\sum_{j\in[n]\setminus\{i\}}\left((1-p_i)\cE_{ij}(f,\log f)+\sum_{k\in [n]\setminus\{i\}}\cE_{kj}(f,\log f)\right)\\
& \le & \frac{\Delta}{\delta^2}\sum_{j\in[n]\setminus\{i\}}\sum_{k=1}^n\cE_{kj}(f,\log f).
\end{eqnarray}
Summing over all sites $i\in[n]$ and recalling (\ref{dec:E}), we conclude that
\begin{eqnarray}
(n-1)\ent(f) & \le & \frac{2\Delta}{\delta^2}(n-1) \cE(f,\log f).
\end{eqnarray}
Since this is true for every $f\colon\Omega\to(0,\infty)$, we have just established a \textsc{MLSI} with constant $\frac{\delta^2}{2\Delta}$ for our $n-$dimensional \textsc{ZRP}. This completes our induction step.

\bibliographystyle{plain}
\bibliography{ZRP}

\end{document}